\documentclass[12pt,reqno]{amsart}
\usepackage{graphicx}
\newtheorem{thm}{Theorem}[section]
\newtheorem{cor}[thm]{Corollary}

\theoremstyle{definition}

\theoremstyle{remark}

\numberwithin{equation}{section}

\begin{document}

\title[]{ bounds for the gamma function}%
\author{Necdet Bat{\i}r}
\address{department of mathematics\\
faculty of sciences and arts\\
nev{\c{s}}eh{\i}r hac{\i} bekta{\c{s}} veli university, nev{\c{s}}eh\i r, turkey}
\email{nbatir@hotmail.com , nbatir@nevsehir.edu.tr}


\date{November 10, 2016}
\begin{abstract} We improve the upper bound of the following inequalities for the gamma function $\Gamma$ due to H. Alzer and the author.
\begin{equation*}
\exp\left(-\frac{1}{2}\psi(x+1/3)\right)<\frac{\Gamma(x)}{x^xe^{-x}\sqrt{2\pi}}<\exp\left(-\frac{1}{2}\psi(x)\right).
\end{equation*}
We also  prove the following new inequalities: For $x\geq1$
\[
\sqrt{2\pi}x^xe^{-x}\left(x^2+\frac{x}{3}+a_*\right)^{\frac{1}{4}}<\Gamma(x+1)<\sqrt{2\pi}x^xe^{-x}\left(x^2+\frac{x}{3}+a^*\right)^{\frac{1}{4}}
\]
with  the best possible constants $a_*=\frac{e^4}{4\pi^2}-\frac{4}{3}=0.049653963176...$, and $a^*=1/18=0.055555...$,
and for $x\geq0$
\begin{equation*}
\exp\left[x\psi\left(\frac{x}{\log (x+1)}\right)\right]\leq\Gamma(x+1)\leq\exp\left[x\psi\left(\frac{x}{2}+1\right)\right],
\end{equation*}
where $\psi$ is the digamma function.
\end{abstract}
\maketitle
\section{Introduction }
For a positive real number $x$ the gamma function $\Gamma$  is defined by
$$\Gamma (x)=\int_0^\infty {u^{x - 1}e^{ - u}du}.$$
The most important function related to the gamma function is the digamma or psi function, which is defined by as the  logarithmic derivative of $\Gamma$, that is, $\psi(x)=\Gamma'(x)/\Gamma(x), x>0$. Furthermore, the derivatives $\psi', \psi'',...,$ are called the polygamma functions. The gamma function plays a very important role in many branches of mathematics and science. It has applications in the theory of special functions, number theory, physics and statistics. A detailed description of the history and development of this function can be found in \cite{12}.  In the last decade there has been intensive interest for the gamma function, and  it has been published many remarkable inequalities for it by many authors. For an overview of approximations for the gamma function, the readers are referred to [1-6,8-11,13,18]  and bibliographies in these papers. The gamma function satisfies the functional equation $\Gamma(x+1)=x\Gamma(x)$, and has the following canonical product representation
\begin{equation*}
\Gamma(x+1)=e^{-\gamma x}\prod_{k=1}^{\infty}\frac{k}{k+x}e^{x/k}\quad  x>-1,
\end{equation*}
where $\gamma=0.57721...$ is the Euler-Mascheroni constant; see \cite[pg.346]{14}.
Taking the logarithm of both sides of this formula, we obtain for $x>-1$
\begin{equation}\label{e:1}
\log\Gamma(x+1)=-\gamma x+\sum\limits_{k=1}^\infty\left[\frac{x}{k}-\log(x+k)+\log k\right].
\end{equation}
Differentiation gives
\begin{equation}\label{e:2}
\psi(x+1)=-\gamma+\sum\limits_{k=1}^\infty\left[\frac{1}{k}-\frac{1}{k+x}\right] \quad x>-1.
\end{equation}

In this work we continue studying this issue. In \cite{4} Alzer and the author studied monotonicity property of the function
\begin{equation*}
G_c(x)=\log\Gamma(x)-x\log x+x-\frac{1}{2}\log (2\pi)+\frac{1}{2}\psi(x+c),
\end{equation*}
and they proved that $G_a(x)$ is completely monotonic on $(0,\infty)$ if and only if $a\geq1/3$, while $-G_b(x)$ is completely monotonic if and only if $b=0$. Consequently, they proved for $x>0$ that
\begin{equation}\label{e:3}
\exp\left(-\frac{1}{2}\psi(x+1/3)\right)<\frac{\Gamma(x)}{x^xe^{-x}\sqrt{2\pi}}<\exp\left(-\frac{1}{2}\psi(x)\right),
\end{equation}
In 2012 $q$ analogue of these inequalities has been obtained by A. Salem \cite{16}. An interesting improvement of them can be found in \cite{15}. The lower bound here is extremely accurate  but the same thing can not be said for the upper bound. Our first aim in this work is to improve the upper bound and to replace it by a much better one. We recall that a function $f$ is said to be completely monotonic on an interval $I$ if it has derivatives of all orders on $I$ and  $(-1)^nf^{(n)}(x)\geq0$ for all $x\in I$ and n=0,1,2,... ; see \cite{19}.

In \cite{6} the author proved the following asymptotic formula for the factorial function
\[
n!\approx \sqrt{2\pi}n^ne^{-n}\left(n^2+\frac{n}{3}+\frac{1}{18}-\frac{2}{405n}-\frac{31}{9720n^2}+\cdots.\right)^{\frac{1}{4}}.
\]
Our second aim is motivated by this formula and we intend  to determine the largest real number $a_*$ and the smallest real number $a^*$ such that the following inequalities hold for all $x\geq 1$.
\begin{equation*}
\sqrt{2\pi}x^xe^{-x}\left(x^2+\frac{x}{3}+a_*\right)^{\frac{1}{4}}<\Gamma(x+1)<\sqrt{2\pi}x^xe^{-x}\left(x^2+\frac{x}{3}+a^*\right)^{\frac{1}{4}}.
\end{equation*}
Our third aim here is to provide elegant and new lower and upper bounds for the gamma function in terms of the digamma function $\psi.$ More precisely, we prove that
\begin{equation*}
\exp\left[x\psi\left(\frac{x}{\log (x+1)}\right)\right]\leq\Gamma(x+1)\leq\exp\left[x\psi\left(\frac{x}{2}+1\right)\right], x\geq0.
\end{equation*}
It is worth to note that these bounds are interesting because they don't contain the terms $\sqrt{2\pi}$, $x^x$ and $e^{-x}$ as being usually.
The algebraic and numerical computations have been carried out with the aid of computer software \textit{Mathematica 10}.

We are ready to present our main results.

\section{Main results}
Our first theorem improves the upper bound given in (\ref{e:3}).
\begin{thm}Let $x>0$. Then we have
\begin{equation*}
\exp\bigg\{-\frac{1}{2}\psi(\delta_*(x))\bigg\}<\frac{\Gamma(x)}{x^xe^{-x}\sqrt{2\pi}}<\exp\bigg\{-\frac{1}{2}\psi(\delta^*(x))\bigg\},
\end{equation*}
where
\begin{equation*}
\delta_*(x)=x+\frac{1}{3} \quad \mbox{and} \quad \delta^*(x)=\frac{1}{2}\frac{1}{(x+1)\log(1+1/x)-1}.
\end{equation*}
\end{thm}
\begin{proof}
In \cite[pg. 47, Eq.(42)]{17} it has been recorded that
\begin{equation}\label{e:4}
\int\limits_{0}^x\log\Gamma(u)du=\frac{x(1-x)}{2}+\frac{x}{2}\log(2\pi)-(1-x)\log\Gamma(x)-\log G(x),
\end{equation}
where $G$ is Barnes' $G$-function, which satisfies the functional equation
\begin{equation}\label{e:5}
G(z+1)=\Gamma(z)G(z)\quad \mbox{and}\quad G(1)=1.
\end{equation}
Please refer to \cite[pg.38-61]{17} for details. Using identities (\ref{e:4}) and (\ref{e:5}) we find that
\begin{equation}\label{e:6}
\int\limits_{x}^{x+1}\log\Gamma(u)du=x\log x-x+\frac{1}{2}\log (2\pi).
\end{equation}
Integrating both sides of the equation (\ref{e:1}) over $(x-1,x)$, we get
\begin{align}\label{e:7}
\int\limits_{x}^{x+1}\log\Gamma(u)du&=-\frac{\gamma(2x-1)}{2}+\sum\limits_{k=1}^\infty\bigg[\frac{2x-1}{2k}-(x+k)\log(x+k)\nonumber\\
&+(x-1+k)\log(x-1+k)+\log k+1\bigg].
\end{align}
By Taylor's Theorem, there exists a $\theta(k)$, depending on $x$,  such that $0<\theta(k)<1$ and
\begin{align}\label{e:8}
(x+k)\log(x+k)&-(x+k-1)\log(x+k-1)&\nonumber\\
&=\log(x+k-1)+1+\frac{1}{2}\frac{1}{x+k-1+\theta(k)}.
\end{align}
Therefore (\ref{e:7}) can be written as following:
\begin{align}\label{e:9}
\int\limits_{x}^{x+1}\log\Gamma(u)du&=-\gamma(x-1)+\sum\limits_{k=1}^\infty\bigg[\frac{x-1}{k}-\log(x-1+k)+\log k\bigg]\nonumber\\
&+\frac{1}{2}\bigg[-\gamma+\sum\limits_{k=1}^\infty\left(\frac{1}{k}-\frac{1}{x+k-1+\theta(k)}\right)\bigg].
\end{align}
By (\ref{e:1}) and (\ref{e:6}) this becomes
\begin{equation}\label{e:10}
x\log x-x+\frac{1}{2}\log(2\pi)-\log\Gamma(x)=\frac{1}{2}\bigg[-\gamma+\sum\limits_{k=1}^\infty\left(\frac{1}{k}-\frac{1}{k+x-1+\theta(k)}\right)\bigg].
\end{equation}
From (\ref{e:8}) we have
\begin{equation}\label{e:11}
\theta(k)=\frac{1}{2}\frac{1}{(x+k)\log(x+k)-(x+k)\log(x+k-1)-1}-k-x+1.
\end{equation}
So in order to prove that $\theta$ is strictly increasing on $[1,\infty]$, we only need to show that
\begin{equation*}
f(u):=\frac{1}{2}\frac{1}{(u+1)\log(u+1)-(u+1)\log u-1}-u
\end{equation*}
is strictly increasing on $(0,\infty)$. Setting $u=1/t$, we get
\begin{equation*}
f(1/t)=\frac{1}{2}\frac{t}{(t+1)\log(t+1)-t}-\frac{1}{t}.
\end{equation*}
Differentiating gives
\begin{equation}\label{e:12}
-\frac{1}{t^2}f'(1/t)=\frac{h(t)}{2t^2((t+1)\log(t+1)-t)^2},
\end{equation}
where
\begin{equation*}
h(t)=t^2\log(t+1)-t^3+2((t+1)\log(t+1)-t)^2.
\end{equation*}
Differentiating three times, we get
\begin{equation*}
(1+t)h'''(t)=8\log(1+t)-\frac{6t^3+12t^2+8t}{(t+1)^2}=g(t), \quad \mbox{say}.
\end{equation*}
\begin{equation*}
  g'(t)=-\frac{6t^3+10t^2}{(1+t)^3}<0,
\end{equation*}
that is, $h'''(t)<0.$ Since $h(0)=h'(0)=h''(0)=0$, this implies that $h(t)<0$ for $t>0$. Thus we conclude from (\ref{e:12}) that $\theta$ is strictly increasing on $[1,\infty)$. By applying L'Hospital rule we can easily compute that
\begin{equation}\label{e:13}
\theta(\infty):=\lim\limits_{k\to\infty}\theta(k)=\lim\limits_{u\to\infty}f(u)=\lim\limits_{t\to0}f(1/t)=\frac{1}{3}.
\end{equation}
Also, from (\ref{e:11}) it is clear that
\begin{equation}\label{e:14}
\theta(1)=\frac{1}{2}\frac{1}{(x+1)\log(x+1/x)-1}-x.
\end{equation}
Utilizing (\ref{e:10}), as a  direct consequence of the fact that $\theta$ is strictly increasing, we obtain
\begin{align*}
-\frac{1}{2}\bigg[-\gamma&+\sum\limits_{k=1}^\infty\left(\frac{1}{k}-\frac{1}{k+x-1+\theta(\infty)}\right)\bigg]<\log\Gamma(x)-x\log x+x\\
&-\frac{1}{2}\log(2\pi)<-\frac{1}{2}\bigg[-\gamma+\sum\limits_{k=1}^\infty\left(\frac{1}{k}-\frac{1}{k+x-1+\theta(1)}\right)\bigg].
\end{align*}
By (\ref{e:2}) this is equivalent to
\begin{equation*}
\exp\bigg\{-\frac{1}{2}\psi(x+\theta(\infty))\bigg\}<\frac{\Gamma(x)}{x^xe^{-x}\sqrt{2\pi}}<\exp\bigg\{-\frac{1}{2}\psi(x+\theta(1))\bigg\}.
\end{equation*}
Taking into account identities (\ref{e:13}) and (\ref{e:14}) this completes the proof of Theorem 2.1.
\end{proof}
Our next theorem provides new bounds for the gamma function.
\begin{thm}Let $x\geq 1$. Then we have
\begin{equation}\label{e:15}
    \sqrt{2\pi}x^xe^{-x}\left(x^2+\frac{x}{3}+a_*\right)^{\frac{1}{4}}\leq \Gamma(x+1)<\sqrt{2\pi}x^xe^{-x}\left(x^2+\frac{x}{3}+a^*\right)^{\frac{1}{4}},
\end{equation}
where the constants $a_*=\frac{e^4}{4\pi ^2}-\frac{4}{3}=0.049653963176...$ and $a^*=\frac{1}{18}=0.05555...$ are the best possible.
\end{thm}
\begin{proof} The double inequality (\ref{e:15}) can be written as
\begin{equation}\label{e:16}
    a_*\leq \left(\frac{\Gamma(x+1)}{\sqrt{2\pi}x^xe^{-x}}\right)^4-x^2-\frac{x}{3}<a^*.
\end{equation}
In order to prove (\ref{e:15}) we define for $x>0$: $$\phi(x)=\left(\frac{\Gamma(x+1)}{\sqrt{2\pi} x^{x}e^{-x}}\right)^4-x^2-\frac{x}{3}.$$
If we differentiate, we get $$\phi'(x)=\frac{(\Gamma(x+1))^4(1/x-(\log x-\psi(x)))}{\pi^2x^{4x}e^{-4x}}-2x-\frac{1}{3}.$$
Applying the inequalities
\begin{align}\label{e:17}
\sqrt{\pi}x^xe^{-x}&\left(8x^3+4x^2+x+\frac{1}{100}\right)^\frac{1}{6}<\Gamma(x+1)\notag\\
&<\sqrt{\pi}x^xe^{-x}\left(8x^3+4x^2+x+\frac{1}{30}\right)^\frac{1}{6},
\end{align}
see \cite{2}, and
\begin{align*}
\log x-\psi(x)<\frac{1}{2x}&-\frac{1}{12x^2}+\frac{1}{120x^4}-\frac{1}{252x^6}+\frac{1}{240x^8}\\
&-\frac{1}{132x^{10}}+\frac{691}{32760x^{12}}-\frac{1}{12x^{14}},
\end{align*}
see \cite[Theorem 8]{3}, we get $$\phi'(x)>\left(8x^3+4x^2+x+\frac{1}{100}\right)^{\frac{2}{3}}p_1(x)-p_2(x)(2x+1/3).$$
where
\begin{align*}
p_1(x)=&-60060+15202x^2-5460x^4+3003x^6-2860x^8+6006x^{10}\\
&-60060x^{12}+360360x^{13}
\end{align*}
and $$p_2(x)=720720x^{14}.$$
In view of this inequality in order to show that  $\phi$ is strictly increasing for $x\geq 1$ it is enough to see
\begin{equation}\label{e:18}
    \left(8x^3+4x^2+x+\frac{1}{100}\right)^2(p_1(x))^3-(p_2(x))^3(2x+1/3)^3>0
\end{equation}
for $x\geq 1$. If we expand the polynomial on the left hand side into its Taylor series at x=1, we obtain
\begin{align*}
 &\left(8x^3+4x^2+x+\frac{1}{100}\right)^2(p_1(x))^3-(p_2(x))^3(2x+1/3)^3\\
 &>1221+4595(x-1)+8755(x-1)^2+1115(x-1)^3+1061(x-1)^4\\
 &+8004(x-1)^5+4961(x-1)^6+2589(x-1)^7+1158(x-1)^8\\
 &+4505(x-1)^9+\cdots +8135(x-1)^{24}+6374(x-1)^{25}+4561(x-1)^{26}\\
 &+2977(x-1)^{27}+\cdots+2416(x-1)^{38}+3636(x-1)^{39}+4445(x-1)^{40}\\
 &+4242(x-1)^{41}+2963(x-1)^{42}+1347(x-1)^{43}+2994(x-1)^{44}\\
 \end{align*}
which is positive since all the coefficients are positive. We therefore have $\phi'(x)>0$ for $x\geq 1$. From increasing  monotonicity of $\phi$ on $[1,\infty)$, we conclude that
 $$\phi(1)=\frac{e^4}{4\pi^2}-\frac{4}{3}\leq\phi(x)<\mathop {\lim }\limits_{x \to \infty }\phi(x).$$
 It remains to prove that $\mathop {\lim }\limits_{x \to \infty }\phi(x)=1/18.$ By the help of  (\ref{e:17}) we find that
\begin{align*}
\left(8x^3+4x^2+x+\frac{1}{100}\right)^\frac{2}{3}&-2x-1/3<\phi(x)\\
&<\left(8x^3+4x^2+x+\frac{1}{30}\right)^\frac{2}{3}-2x-1/3.
\end{align*}
It can be easily shown that the limits of both of the bounds here tend to 1/18 as $x$ approaches $\infty$. This completes the proof of Theorem 2.2.
\end{proof}

\begin{thm}For all $x>0$, we have
\begin{equation}\label{e:19}
\alpha_* .x^xe^{-x}\left(x^2+\frac{x}{3}+\frac{1}{18}\right)^{\frac{1}{4}}<\Gamma(x+1)<\alpha^* .x^xe^{-x}\left(x^2+\frac{x}{3}+\frac{1}{18}\right)^\frac{1}{4},
\end{equation}
where $\alpha_*=\sqrt[4]{18}/\sqrt{2\pi}=0.821728...$  and $\alpha^*=\sqrt{2\pi}=2.50663...$ are the best possible constants.
\end{thm}
\begin{proof}For $x\geq 0$, we define
\begin{equation*}
\Theta(x)=\log(\Gamma(x+1))-x\log x+x-\frac{1}{2}\log 2\pi-\frac{1}{4}\log \left(x^2+\frac{x}{3}+\frac{1}{18}\right).
\end{equation*}
If we differentiate we find that
\[
\Theta'(x)=\psi(x+1)-\log x-\frac{18x+3}{36x^2+12x+2}
\]
and
\[
\Theta''(x)=\psi'(x+1)-\frac{1}{x}+\frac{162x^2+54x}{(18x^2+6x+1)^2}.
\]
By using the well known functional relation
$$
\psi'(x+1)-\psi'(x)=-\frac{1}{x^2}
$$
we obtain
\begin{equation*}
\Theta''(x+1)-\Theta''(x)=\frac{p(x)}{q(x)},
\end{equation*}
where
\[
p(x)=625 + 9816 x + 42516 x^2 + 63936 x^3 + 38556 x^4 + 7776 x^5>0
\]
and
\[
q(x)=x (1 + x)^2 (1 + 6 x + 18 x^2)^2 (25 + 42 x + 18 x^2)^2>0.
\]
This shows that $\Theta''(x+1)-\Theta''(x)>0$ for $x\geq0$. By mathematical induction we have $\Theta''(x)<\Theta''(x+n)$ for all $n\in \mathbb{N}$, and $\Theta''(x)<\lim\limits_{n\to\infty}\Theta''(x+n)=0$, that is, $\Theta'$ is strictly decreasing on $(0,\infty)$. Similarly, $\Theta'(x)>\Theta'(x+n)>\lim\limits_{n\to\infty}\Theta'(x+n)=0$, since $\lim\limits_{x\to\infty}[\log x-\psi(x)]=0$; see \cite{3}. This reveals that $\Theta$ is strictly inceasing. Applying the classical Stirling formula we find
$\lim\limits_{x\to\infty}\Theta(x)=0$, we therefore conclude that
\[
\frac{1}{4}\log 18-\frac{1}{2}\log 2\pi=\Theta(0)\leq \Theta(x)<\mathop {\lim }\limits_{x \to \infty }\Theta(x)=0,
\]
which is equivalent with (\ref{e:19}).
\end{proof}
The following is a natural consequence of the fact that $\Theta$ is strictly increasing and  $\Theta(1)=1-\frac{1}{2}\log (2\pi)-\frac{1}{4}\log (\frac{25}{18})$.
\begin{cor} For all natural number $n$, we have
\[
\beta_*\cdot n^ne^{-n}\left(n^2+\frac{n}{3}+\frac{1}{18}\right)^{\frac{1}{4}}<n!\leq\beta^*\cdot n^ne^{-n}\left(n^2+\frac{n}{3}+\frac{1}{18}\right)^\frac{1}{4},
\]
where the scalers $\beta_*=\frac{e}{\sqrt{2\pi}}\left(\frac{18}{25}\right)^{\frac{1}{4}}=0.998936...$ and $\beta^*=\sqrt{2\pi}=2.50663...$ are best possible.
\end{cor}
Our last theorem gives elegant bounds for the gamma function in terms of the digamma function.
\begin{thm} For $x>0$ we have
\begin{equation*}
\exp\left[x\psi\left(\frac{x}{\log(x+1)}\right)\right]\leq\Gamma(x+1)\leq\exp\left[x\psi\left(\frac{x}{2}+1\right)\right].
\end{equation*}
\end{thm}
\begin{proof} Applying the the Mean Value Theorem for differentiation  we get
\begin{equation}\label{e:20}
\log(x+k)-\log k=\frac{x}{k+\varphi(k)},\quad 0<\varphi(k)<x.
\end{equation}
Therefore (\ref{e:1}) can be written as follows
\begin{equation}\label{e:21}
\log \Gamma(x+1)=x\bigg[-\gamma+\sum\limits_{k=1}^\infty\left(\frac{1}{k}-\frac{1}{k+\varphi(k)}\right)\bigg]
\end{equation}
From (\ref{e:20}) we obtain
\begin{equation*}
\varphi(k)=\frac{x}{\log (1+x/k)}-k.
\end{equation*}
Differentiation gives
\begin{equation*}
\varphi'(k)=\frac{\frac{x^2}{k^2+kx}-\log^2\left(1+\frac{x}{k}\right)}{\log^2\left(1+\frac{x}{k}\right)}.
\end{equation*}
It follows that $\varphi'(k)>0$ if and only if
\begin{equation*}
\sqrt{k(k+x)}<\frac{(k+x)-k}{\log(k+x)-\log k},
\end{equation*}
which follows from the familiar geometric-logarithmic mean inequality $G\leq L$; see \cite[pg.134]{7}. Hence, $\varphi$ is strictly increasing on $[1,\infty)$. We have
\begin{equation*}
\varphi(\infty):=\lim\limits_{k\to\infty}\varphi(k)=\frac{x}{2}\quad \mbox{and} \quad \varphi(1)=\frac{x}{\log(x+1)}-1.
\end{equation*}
We therefore conclude from (\ref{e:21}) that for $x>0$
\begin{align*}
x\bigg[-\gamma+&\sum\limits_{k=1}^\infty\left(\frac{1}{k}-\frac{1}{k+\varphi(1)}\right)\bigg]<\log\Gamma(x+1)\\
&< x\bigg[-\gamma+\sum\limits_{k=1}^\infty\left(\frac{1}{k}-\frac{1}{k+\varphi(\infty)}\right)\bigg],
\end{align*}
which is equivalent to
\begin{equation*}
x\psi(\varphi(1)+1)<\log\Gamma(x+1)<x\psi(\varphi(\infty)+1),
\end{equation*}
completing  the proof of  Theorem 2.5.
\end{proof}
\textbf{Acknowledgement.} The author wishes  to thank an anonymous referee for reading the manuscript carefully and detecting many mistakes, which have improved the presentation of the paper.

\end{document}